\documentclass{article}
\usepackage{amsmath, amsfonts, amsthm, amssymb}
\usepackage{enumerate}
\usepackage{hyperref}

\theoremstyle{plain}
\newtheorem{proposition}{Proposition}[section]
\newtheorem{theorem}[proposition]{Theorem} 
\newtheorem{lemma}[proposition]{Lemma}
\newtheorem{conjecture}[proposition]{Conjecture}
\newtheorem{corollary}[proposition]{Corollary}

\theoremstyle{definition}

\newtheorem{remark}[proposition]{Remark}

\newcommand{\bra}{\langle}
\newcommand{\ket}{\rangle}

\DeclareMathOperator{\ch}{char}

\DeclareMathOperator{\F}{\mathbb{F}}

\DeclareMathOperator{\Stab}{Stab}

\def\sep{{\mathrm{sep}}}

\begin{document}

\title{The Noether number   of the non-abelian \\ group of order $3p$}   
\author{K\'alm\'an Cziszter  
\thanks{This paper is based on results from the PhD thesis of the author written at the Central European University.} 
}
\date{}
\maketitle 
{\small \begin{center} 
Central European University, Department of Mathematics and its Applications, 
N\'ador u. 9, 1051 Budapest, Hungary \\
Email: cziszter\_kalman-sandor@ceu-budapest.edu 
\end{center}
}

\maketitle

\begin{abstract} 
It is proven that for any representation over a field of characteristic $0$ 
of the non-abelian semidirect product of a cyclic group of prime order $p$ and the group of order $3$  
the corresponding algebra of polynomial invariants is generated by elements of degree at most $p+2$. 
We also determine the exact degree bound for any separating system of the polynomial invariants
of any representation of  this group in characteristic not dividing $3p$.
\vskip 0.5 cm


\end{abstract} 

\section{Introduction}\label{sec:intro} 

A classical theorem of Noether asserts that the ring of polynomial invariants of a finite group $G$ is generated by its elements of degree  at most 
$|G|$, provided that the characteristic of the base field $\F$ does not divide $|G|$ (this was proved in \cite{noether:1916} for $\ch(\F)=0$, 
and the result was extended more recently to non-modular positive characteristic independently by Fleischmann \cite{fleischmann} and Fogarty \cite{fogarty}). 
Denoting by $\beta(\F[V]^G)$ the minimal positive integer $n$ such that the algebra $\F[V]^G$ of polynomial invariants is generated by elements of degree at most $n$, 
 the \emph{Noether number} is defined as
\[\beta(G):=\sup_V\beta(\F[V]^G)\]
where $V$ ranges over all finite dimensional $G$-modules over $\F$. Throughout the paper we shall assume that $\ch(\F)$ does not divide $|G|$.  
For a cyclic group $Z_n$ of order $n$  we have $\beta(Z_n)=n$, but Noether's bound is never sharp for a non-cyclic group,  see \cite{schmid}, \cite{domokos-hegedus}, 
\cite{sezer}, \cite{cziszter-domokos:1}.  These latter works show that an unavoidable step in improving the Noether bound is to find good upper bounds for 
$\beta(Z_p\rtimes Z_q)$, where $Z_p\rtimes Z_q$ is the non-abelian semidirect product of cyclic groups of odd prime order (hence  $q$ divides $p-1$). 
The exact value of the Noether number is known only for  a few very special series of groups (see \cite{cziszter-domokos:2}) and for a couple of groups of small order.  
The following conjecture  is  attributed to \cite{pawale} in \cite{wehlau}: 

\begin{conjecture}[Pawale] \label{conj:pawale}  We have  $\beta(Z_p\rtimes Z_q)=p+q-1$. 
\end{conjecture} 
The analogous statement for $q=2$ (i.e. when the group is the dihedral group of order $2p$) is proved in \cite{schmid} for $\ch(\F)=0$ and in \cite{sezer} for the case when $\ch(\F)$ does not divide $2p$. Otherwise the only known case of Conjecture~\ref{conj:pawale} is that $\beta(Z_7\rtimes Z_3)=9$ (this was proved for $\ch(\F)=0$ in \cite{pawale}, and an alternative proof extending to non-modular positive characteristic is given in \cite{cziszter-domokos:1}). The inequality $\beta(Z_p\rtimes Z_q)\ge p+q-1$ follows  from a more general statement in \cite{cziszter-domokos:2}. 
Explicit upper bounds for $\beta(Z_p\rtimes Z_q)$ are given in \cite{cziszter-domokos:1}; in the special case $q=3$ they yield 
$\beta(Z_p\rtimes Z_3)\le p+6$ (this was proved in \cite{pawale} for $\ch(\F)=0$).  

This paper focuses on the non-abelian semidirect product $G:=Z_p\rtimes Z_3$.
Our main result confirms Conjecture~\ref{conj:pawale} for this group $G$ inasmuch as we prove that
$\beta(\F[V]^G)=p+2$ holds for $\ch(\F) = 0$ or for  any multiplicity free $G$-module $V$ in non-modular positive characteristic. 
In fact this is the special case $k=1$  of   Theorem~\ref{betak_zpz3}, asserting that $\beta_k(\F[V]^G)=kp+2$ for any positive integer $k$, 
where $\beta_k(\F[V]^G)$  denotes the maximal degree of a homogeneous element of $\F[V]^G$ not contained in the $(k+1)$-st power of the maximal homogeneous ideal $\F[V]^G_+$. 

As an application of the result on multiplicity free modules
we determine the exact value of $\beta_{\sep}(Z_p\rtimes Z_3)$, the version of the Noether number for \emph{separating invariants}, studied recently in 
\cite{kohls-kraft}. 
Thanks to a theorem of Draisma, Kemper, and Wehlau \cite{draisma-kemper-wehlau} on \emph{cheap polarization} (see also \cite{losik-michor-popov} and  \cite{grosshans:2007}), our study of the multiplicity free $Z_p\rtimes Z_3$-modules will be sufficient to determine the \emph{separating Noether bound} for arbitrary modules, see Theorem~\ref{thm:sepZpZ3}.    


\section{Preliminaries}\label{sec:prel}

In the rest of the paper $G$ will be the non-abelian semidirect product 
\[G:=Z_p\rtimes Z_3=\bra c,d: c^p=d^3=1, dcd^{-1}=c^r\ket\]  
where $r$ has order $3$ modulo $p$ for some prime $p$ with $3$ dividing $p-1$, and $\F$ is a  field of characteristic different from $p$ or $3$.  For sake of convenience we shall assume in the text that $\F$ is algebraically closed. This is relevant for example when we list the irreducible $G$-modules. However, the Noether number or the condition that a $G$-module is multiplicity free are not sensible for extension of the base field. In particular, the final statements 
Corollary~\ref{cor:zpz3},  Theorem~\ref{betak_zpz3},  and Theorem~\ref{thm:sepZpZ3} remain valid without the assumption on algebraic closedness. 
By a \emph{$G$-module} we mean a finite dimensional $\F$-vector space endowed with a linear (left) action of $G$ on $V$, and write $\F[V]$ for the coordinate ring of $V$. 
This is the symmetric tensor algebra of $V^*$, the dual space of $V$ endowed with the natural right action $x^g(v):=x(gv)$ ($g\in G$, $v\in V$, $x\in V^*$). 
The corresponding \emph{algebra of polynomial invariants} is 
\[\F[V]^G:=\{f\in \F[V]\mid f^g=f \qquad \forall g\in G\}.\] 
Denote by $A$ the $p$-element normal subgroup of $G$, and $\hat A:=\hom_{\mathbb{Z}}(A,\F^\times)$ the group of characters of $A$; there is a non-canonic isomorphism between $\hat A$ and $A$. Denote by $B$ a $3$-element subgroup in $G$. The conjugation action of $B$ on $A$ induces an action on $\hat A$ given by 
$\chi^g(a)=\chi(gag^{-1})$, where $\chi\in \hat A$, $g\in B$, $a\in A$. The trivial character of $A$ is fixed by all elements in $B$, whereas the remaining characters are partitioned into  $l:=\frac{p-1}3$ $B$-orbits $O_1,\dots,O_l$ of size $3$. 
Up to isomorphism there are three $1$-dimensional $G$-modules, namely the three $1$-dimensional $B\cong G/A$-modules lifted to $G$. 
The remaining irreducible $G$-modules $V_1,...,V_l$ are in a natural bijection with the $l$-element set $\{O_1,\dots,O_l\}$. Namely the $3$-dimensional  $G$-module induced from a non-trivial character of $A$ is irreducible, and two such induced $G$-modules are isomorphic if and only if the $A$-characters we started with are in the same $B$-orbit $O_i$. 
It follows that it is possible to choose a  set of variables in the polynomial ring $\F[V]$ such that the variables are $A$-eigenvectors permuted by $G$ up to non-zero scalar multiples, moreover, the $B$-orbit of any variable spans an irreducible $G$-module summand in $V^*$. 
We shall always assume that our variables in $\F[V]$ satisfy these requirements. 
Then any monomial $m$ has a \emph{weight} $\theta(m)\in \hat A$ defined by $m^a=\theta(m)(a)m$ for all $a\in A$, and a \emph{weight sequence} 
$\Phi(m):=(\theta(x_1),\theta(x_2),\dots,\theta(x_d))$, where $m=x_1x_2\dots x_d$, and $x_1,x_2,\dots,x_3$ are 
not necessarily different variables; 
so $\Phi(m)$ is a sequence over the abelian group $\hat A$. We shall write $\hat A$ additively. 

We need to recall some terminology and facts about zero-sum sequences (see \cite{geroldinger-gao}, \cite{geroldinger-halterkoch} as a general reference). 
By a \emph{sequence over an abelian group} $C$ (written additively) we mean a sequence  
$S:=(c_1,\ldots,c_n)$ of elements  $c_i\in C$ where repetition of elements is allowed and their order is disregarded. 
The \emph{length} of $S$ is $|S|:=n$. 
By a \emph{subsequence} of $S$ we mean $S_J := (c_j\mid j\in J)$ for some subset $J\subseteq \{1,\ldots,n\}$. 
We write 
$S=S_1S_2$ if $S$ is the concatenation of its subsequences $S_1$, $S_2$. 
The set of \emph{partial sums} of $S$ is 
$\Sigma(S) := \{ \sum_{i\in I} c_i: I  \subseteq \{1,...,n \} \}$. 
We say that $S$ is a \emph{zero-sum sequence} if $c_1+\dots+c_n=0$. 
A zero-sum sequence is \emph{irreducible} if there are no non-empty zero-sum sequences $S_1$ and $S_2$ such that $S=S_1S_2$.  
Furthermore, $S$ is \emph{zero-sum free} if it has no non-empty zero-sum subsequence. 
We call the \emph{height of} $S$ the maximal multiplicity of an element in $S$. 
Given an element $c\in C$ and a positive integer $r$, write $(c^r)$ for the sequence of length $r$ in which $c$ occurs with multiplicity $r$. 
We collect for later reference some facts about zero-sum sequences over the cyclic group $Z_p$. 
The Cauchy-Davenport Theorem asserts that for any non-empty subsets $K,L\subseteq Z_p$
\begin{align} \label{cauchy} 
|\{k+l\mid k\in K,l\in L\}|\geq \min\{p,|K|+|L|-1\}
\end{align}
Vosper's theorem (see Theorem 5.9. in \cite{tao})  states
that equality in \eqref{cauchy} implies that $K$ and $L$ are arithmetic progressions of the same step, 
provided that $|K|,|L| \ge 2$ and $|K+L| \le p-2$. 
We shall repeatedly use the following easy consequence of the Cauchy-Davenport Theorem and Vosper's Theorem: 

\begin{lemma} \label{lemma:easy} 
If $S=(s_1,...,s_d)$ is a sequence of non-zero elements over $Z_p$ for a prime $p$ then $|\Sigma(S)| \ge \min \{ p, d+1 \}$. 
Moreover if $p-1 \ge |\Sigma(S)| = d+1$ then  $S= (-a^k,a^{d-k})$ for some $a \in Z_p \setminus \{ 0 \}$ and $0 \le k \le d$.
\end{lemma}

Zero-sum sequences appear in the study of $\F[V]^G$ because there is a graded $\F[V]^G$-module surjection 
\[\tau:\F[V]^A\to \F[V]^G,\qquad f\mapsto \sum_{b\in B}f^b,\] 
and $\F[V]^A$ is spanned as an $\F$-vector space by the monomials $m$ such that $\Phi(m)$ is a zero-sum sequence over $\hat A$. 
The map $\tau$ is called the \emph{relative transfer map} (see for example Chapter 1 in \cite{benson} for its basic properties).

\section{The multiplicity free modules over $Z_p \rtimes Z_3$}

Throughout this section we assume that $V$ is a  \emph{multiplicity free} $G$-module; 
that is, $V= V_1 \oplus ... \oplus V_n$  is the direct sum of the pairwise non-isomorphic irreducible  $G$-modules $V_i$. 
This  direct decomposition gives an identification  $\F[V]=\F[V_1]\otimes \dots\otimes \F[V_n]$ 
and accordingly any monomial  $m\in \F[V]$ has a canonic factorization $m=m_{V_1}...m_{V_n}$ into monomials $m_{V_i} \in \F[V_i]$. 
We  write $\deg_{V_i}(m) := \deg(m_{V_i})$ for each $i = 1,...,n$.
Let $R:=\F[V]^G$ and $I:=\F[V]^A$. 
The degree $d$ homogeneous component of the graded algebras $R$, $I$
is denoted by $R_d$, $I_d$ and  
their maximal homogenous ideals by  $R_+$ and $I_+$, which are spanned by all  homogeneous components of positive degree. 
Moreover, $(R_+)_{\le p}$ stands for the subspace  $R_1 \oplus ... \oplus R_p$. 
For any subspaces $K,L\subset R$  the $\F$-subspace spanned by the set $\{kl\mid k\in K, l\in L\}$ is denoted by $KL$.

\begin{proposition} \label{ZpZ3_borzalom}
Let $V=V_1\oplus...\oplus V_n$ where each $V_i$ is a $3$-dimensional  irreducible $G$-module.
If $m \in I$ is a monomial such that $\deg(m) \ge p+1$ and $\deg_{V_i}(m) \ge 4$ for some $1 \le i \le n$ 
 then $m \in I_+(R_+)_{\le p}$. 
\end{proposition}

\begin{proof} 
Let $t$ be the index for which $\deg_{V_t}(m)$ is maximal; then $\deg_{V_t}(m) \ge 4$. 
Let $\F[V_t] = \F[x,y,z]$ where the variables $x,y,z$ are $A$-eigenvectors  cyclically permuted by $B$. 
Note that $xyz$ is a $G$-invariant. 
For a monomial $m\in \F[V]$ let $\lambda(m)$ denote the non-increasing sequence of integers 
$(\lambda(m)_1,\dots,\lambda(m)_h)$ where $\lambda(m)_i$ is the number of elements of $\hat A$ 
which have multiplicity at least $i$ in the weight sequence $\Phi(m)$ 
(here $h$ is the maximal multiplicity of a weight in $\Phi(m)$).  
Choose a divisor $w\mid m_{V_t}$  such that $\deg(w)=4$ and $\lambda(w)$ is minimal possible with respect to the anti-lexicographic ordering. 
We have several cases:

(i) If $\lambda(w) =(3,1)$ then $w$ contains  $xyz \in R_3$ hence $m \in I_+R_3$

(ii) If $\lambda(w) = (2,2)$, say $w=x^2y^2$ and $- \theta(xy) \in \Sigma(\Phi(m/w))$ 
then we can find an $A$-invariant monomial $v$  such that $xy \mid v \mid m$ and
the zero-sum sequence $(\theta(xy))\Phi(v/xy)$ is irreducible. We claim moreover that $\deg(v) \le p$. 
For otherwise $\deg(v) = p+1$ and $\Phi(v/xy)=(c^{p-1})$ where $c = \theta(xy)$. 
By the maximality of $\deg_{V_t}(m)$ this weight $c$ cannot belong to any  irreducible component different from $V_t$. 
Moreover, by the minimality of $\lambda(w)$ the weight $c$ must coincide with $\theta(x)$ or $\theta(y)$. 
But then either $\theta(y)=0$ or $\theta(x)=0$, respectively, which is a contradiction. 
Now set $u=m/v$ and observe that $u \tau(v) \in I_+(R_+)_{\le p}$ while 
$m - u\tau(v) \in I_+R_3$ by case (i).

(iii) If $\lambda(w) = (2,2)$, say $w=x^2y^2$ but $- \theta(xy) \not\in \Sigma(\Phi(m/w))$;
 suppose in addition that  $S \cap -S \neq \emptyset$ where $S \subseteq \hat{A}$ denotes the set of weights occurring in $\Phi(m/w)$: 
 then $\Phi(m/w)$ contains a subsequence $(s,-s)T$ where $s \in \hat{A}\cong Z_p$ and $|T| = p-5$.
Given that $|\Sigma(T)|\ge p-4$ by Lemma~\ref{lemma:easy} and $-\theta(xy)\notin \Sigma(T)$,  
at least one of the weights $-\theta(x),-\theta(y),-\theta(x^2y),-\theta(xy^2)$  must occur in $\Sigma(T)$. 
Hence  $m=uvw$,   where $u,v,w$ are $A$-invariant, $\Phi(v)=(s,-s)$ and 
$x\mid w$, $xy^2\mid u$ by symmetry in $x$ and $y$. Now as $\deg(vw)\le 2+|T|+3 \le p$ we have the relation
\[ 3(vw - v^gw^{g^2}) = (v-v^g)\tau(w) + (w-w^{g^2})\tau(v) + \tau(vw - vw^g) \in I(R_+)_{\le p}\] 
whence $uvw - uv^gw^{g^2} \in I_+(R_+)_{\le p}$  holds, 
while  $uv^gw^{g^2}$  falls under case (i).

(iv) If $\lambda(w) = (2,2)$, say $w=x^2y^2$, $- \theta(xy) \not\in \Sigma(\Phi(m/w))$ and $S \cap -S = \emptyset$;
then  $|\Sigma(\Phi(m/w))|\le p-1$ so that $|\Phi(m/w)| \in \{p-3,p-2\}$ by Lemma~\ref{lemma:easy} and the assumption on $\deg(m)$.
As $S \cap -S = \emptyset$ we can apply Balandraud's theorem (see Theorem 8 in ~\cite{balandraud})
stating that $|\Sigma(\Phi(m/w))| \ge 1+\nu_1 + 2\nu_2 + ... +k \nu_k$ where 
$\nu_1 \ge ... \ge \nu_k$ are the multiplicities of the different elements of $\hat A$ occurring in $\Phi(m/w)$. 
Given that $|\Sigma(\Phi(m/w))| \le p-1$ this forces that
$\Phi(m/w)$ must have one the following three forms for some elements $a,b \in \hat{A} \cong Z_p$: 
\[(a^{p-2}) \qquad (a^{p-3}) \qquad  (a^{p-4},b). \]
By the maximality of $\deg_{V_t}(m)$ it is evident 
that $a\in\{ \theta(x),\theta(y)\}$ | except if $p=7$ and $|\Phi(m/w)|=4$. 
But this  also holds in this latter case, too, for otherwise if $a\notin \{\theta(x),\theta(y)\}$ then, 
as  $\hat{A} \setminus \{ 0 \}$ consist of only two $B$-orbits for $p=7$ and $S \cap -S = 0$ by assumption,
it is necessary that $a = -\theta(z) = \theta(xy)$; 
since $\Phi(m)$ is a zero-sum sequence this rules out the possibility  $\Phi(m/w) = (a^4)$, 
so it remains that $\Phi(m/w) = (a^3, b)$ where by the same reason we must have $b = -5a = -2\theta(z)$; 
given that a generator $g$ of $B$ operates on  $\hat A$  by multiplication with $2$, 
this implies that $b \in \{ -\theta(x) , - \theta(y) \}$, a contradiction. 
So from now on we may suppose that  $a=\theta(x)$ (since the case $a=\theta(y)$ is similar). 

\begin{enumerate}
\item[(a)] If $\Phi(m/w) = a^{p-2}$ then $\Phi(m)=(a^{p}, \theta(y)^2)$ and as $\Phi(m)$
is a zero-sum sequence it follows that $\theta(y)=0$, a contradiction.

\item[(b)] If $\Phi(m/w) = a^{p-3}$ then  $\Phi(m)= (a^{p-1}, (ra)^2)$ is a zero-sum sequence, where $r$ has order $3$ modulo $p$. 
Consequently $-a+ 2r a = 0 $ whence $2r \equiv 1 $ modulo $p$. Given that $r^3\equiv 1$ modulo $p$ it follows that $p= 7$, 
and $r=4$. Then $m=uv$, where $\Phi(u)=\Phi(v)=(a^3,4a)$, hence $u\tau(v)\in I_+R_4$ and all monomials of $u\tau(v)-m$ fall under case (i), hence belong to $I_+R_3$. 

\item[(c)] If  $\Phi(m/w) = (a^{p-4},b)$ then  $a \neq \theta(xy)$ as $\theta(y) \neq 0$, 
thus the sequence $S:= (a^{p-4},b, \theta(xy))$ has height $h(S) = p-4$.
Therefore a nonempty zero-sum sequence $T \subseteq S$ exists,
for otherwise if $S$ were zero-sum free then  by a result of  Freeze and Smith \cite{freeze-smith} 
we have $|\Sigma(S)| \ge 2|S| -h(S) +1=p+1$, a contradiction.  
Moreover $T$ cannot contain $\theta(xy)$ since we assumed that $-\theta(xy) \not\in \Sigma(\Phi(m/w))$. 
It follows that $T=(a^j, b)$ for some $0 \le j \le p-4$. 
Here  $j\neq 0$ since $b \neq 0$. 
Similarly $j \neq p-4$, for otherwise $\theta(x^2y^2) = 0$ whence $\theta(x) = -\theta(y)$ follows, 
in contradiction with the fact that the variables $x$ and $y$  belong to the same representation $V_t$.
This way we obtained a factorization $m=uv$ where $\Phi(v) = T$ and $u\tau(v) \in I_+(R_+)_{\le p-4}$ while in the same time 
$m-u\tau(v) \in I_+ R_3 + I_+ (R_+)_{\le p}$ by case (ii) or (i).
\end{enumerate}

(v) If $\lambda(w) = (2,1,1)$, say $w=x^3y$: 
Here $\theta(x^2y) \neq \theta(x)$, for otherwise $\theta(x) = - \theta(y)$, a contradiction as above.
Moreover $\theta(xy)$ is different from both $\theta(x^2y)$ and $\theta(x)$, for otherwise  $\theta(x)=0$ or $\theta(y) = 0$. 
Now we have two cases: 

\begin{enumerate}
\item[(a)] If $\Sigma(\Phi(m/w))$ contains $-\theta(x^2y)$ or $-\theta(x)$ then 
we get a $A$-invariant factorization $m=uv$  such that $x^2y \mid u$ and $x \mid v$.
Here we can assure in addition that $v$ is irreducible, so that $\deg(v) \le p$. 
Then $u \tau(v) \in I_+(R_+)_{\le p}$ while the monomials occurring in $m -u\tau(v)$ both fall under case (i) - (iv), and we are done.

\item[(b)]
If however $-\theta(x) , -\theta(x^2y)  \not\in  \Sigma(\Phi(m/w))$ then 
 $|\Sigma(\Phi(m/w))| \le p-2$ while   $|\Phi(m/w)|  \ge p-3$ by assumption.  
In view of Lemma \ref{lemma:easy} this situation  is only possible    if   $|\Phi(m/w)| = p-3$ and $|\Sigma(\Phi(m/w))| = p-2$, 
so that $\Phi(m/w) = ((-a)^i, a^{p-3-i})$ for some   $a\in \hat A \setminus \{ 0\}$ and $0 \le i \le p-3$.
It is also necessary that $-\theta(xy) \in  \Sigma(\Phi(m/w))$,
so a  factorization $m=uv$ exists where $u,v$ are $A$-invariant monomials,  $x^2\mid u$, $xy \mid v$. 
Clearly $u\neq x^2$, $v\neq xy$,   and $\deg(v)\le p$. 
Therefore $u\tau(v) \in I_+(R+)_{\le p}$ and one of the monomials in $m - u\tau(v)$, say $uv^{g}$  falls under case (i) 
while the other one, $m':= uv^{g^2}$ falls under case (v/a), as here $w':= x^3z$ and $\Phi(m'/w')$ does not consist of 
at most two weights which are negatives of each other.
\end{enumerate}

(vi) $\lambda(w) =(1,1,1,1)$, say  $w=x^4$:
then $|\Sigma(\Phi(m/x^2))| =p $ by Lemma~\ref{lemma:easy}, 
so that $- \theta(x) \in \Phi(m/x^2)$ and this gives us then an $A$-invariant factorization $m=uv$
such that $x \mid u$, $x \mid v$, and $v$ is irreducible,
hence $u\tau(v) \in I_+(R_+)_{\le p}$
whereas the monomials in $m - u\tau(v)$ both fall under cases (i)--(iv). 
 \end{proof}

\begin{corollary} \label{cor:zpz3}
Let $V = V_1 \oplus ... \oplus V_l$ be the direct sum of all pairwise non-isomorphic $3$-dimensional irreducible $G$-modules
, where $l=\frac{p-1}{3}$. 
Then $\F[V]^A_+$ as a module over $\F[V]^G$ is generated by  elements of degree at most $p$.
\end{corollary}
\begin{proof}
For any monomial $m \in \F[V]^A$ of degree $\deg(m) \ge p+1$ there must be an index $1 \le t \le l$ 
such that $\deg_{V_t}(m) \ge 4$, hence  Proposition~\ref{ZpZ3_borzalom} applies.
\end{proof}

We turn now to the $G$-module $V^{\oplus 2}$, where $V$ is as in Corollary~\ref{cor:zpz3}. First of all we fix a direct decomposition
\begin{align}\label{V2}
V^{\oplus 2} = U_1 \oplus ... \oplus U_l \quad \text{ where } U_i = V_{i,1} \oplus V_{i,2} \cong V_i^{\oplus 2} \text{ for all } i =1,...,l
\end{align}
Extending our previous conventions, the set of variables in the ring $\F[V^{\oplus 2}]$ is the union of 
the $A$-eigenbases of each $V_{i,j}^*$ permuted cyclically by $B$; 
 the divisors $m_{V_{i,j}} \in \F[V_{i,j}]$ of a monomial $m \in \F[V^{\oplus 2}]$ are defined accordingly.
Finally, the indices $i$ will be chosen for convenience in such a way that 
 the $A$-weights occurring in  $U_i$ and $U_{l-i}$ are the negatives of each other for each $i=1,...,l/2$.

\begin{proposition} \label{ZpZ3_dupla}
Let  $I := \F[V^{\oplus 2}]^A$ and $R := \F[V^{\oplus 2}]^G$. 
For any monomial $m\in I$   with $\deg(m) \ge p+1$ it holds that 
$m \in I_2I_+^2 + I_+(R_+)_{\le p} $,
except if $p=7$ and $\Phi(m)=(a^ 6,3a,5a)$ for some $a \in Z_p \setminus \{ 0\}$.
\end{proposition}

\begin{proof}
If $\deg_{V_{i,j}}(m) \ge 4$ for any $i=1,...,l$ and $j=1,2$
then $m \in I_+(R_+)_{\le p}$ holds  by Proposition~\ref{ZpZ3_borzalom}.
So for the rest we may suppose that $\deg_{V_{i,j}}(m) \le 3$ for each $i,j$, 
hence $\deg_{U_i}(m) \le 6$ for each $i$. 
On the other hand $\deg_{U_t \oplus U_{l-t}}(m) \ge 7$ for some $t = 1,...,\frac{p-1}{6}$
because $\deg(m) \ge p+1$. 
As a result  $\deg_{U_t}(m) \ge 1$ and $\deg_{U_{l-t}} (m) \ge 1$, 
hence by symmetry, we may suppose that $ m_{V_{t,1}}$ and $m_{V_{l-t,1}}$ are both non-constant monomials. 
Here we have three cases:

(i) if $m \in I_2 w$ for a monomial $w \in I_{\ge p-1}$ then by assumption
$\deg_{V_{i,j}}(w) \le 3$ must hold for each $i,j$, too. We claim that in this case $w \in I_+^2$. 
For otherwise $\Phi(w)$ is an irreducible zero-sum sequence of length at least $p-1$, 
hence it is easily seen e.g. using Balandraud's above mentioned theorem that it must be of the form  
$(a^p)$ or $(a^{p-2},2a)$ for some $a \in Z_p \setminus \{ 0 \}$.
Denoting by $s$ the index for which the weight $a$ belongs to the isotypic component $U_s$, 
we will have $\deg_{U_s}(w) \ge p-2 \ge 7$,  a contradiction, provided that $p > 7$. 
If however $p=7$ then observe that $2$ has order $3$ modulo $7$,
hence the weight $2a$ belongs to the same component $U_s$, and since $m \in I_2 w$ we get 
$\deg_{U_s}(m) = \deg_{U_s}(w)+1 \ge 7$, a contradiction again.

(ii) if $m \not\in I_2 I_+$ and $\Phi(m_{U_t})$ contains  two different weights 
    (or $\Phi(m_{U_{l-t}})$ does so, which is a similar case) then 
take divisors $u_0 \mid m_{U_t}$ and $v_0 \mid m_{U_{l-t}}$ such that
$\deg(v_0)=1$, $\deg(u_0) = 2$  and $\Phi(u_0)$ contains two different weights, too. Set $w := m/u_0v_0$;
we claim that $-\theta(v_0) \in \Sigma(\Phi(w))$, for otherwise $|\Sigma(\Phi(w))| \le p-1$ while $|\Phi(w)| \ge p-2$ 
which is only possible by Lemma~\ref{lemma:easy}  if $\Phi(w) = (-a^k, a^{p-2-k})$ for some $a \in Z_p \setminus \{ 0 \}$ and $0\le k \le p-2$.
But here we must have $k=0$ (or $k=p-2$) as otherwise we would get back to case (i). 
Then for the isotypic component $U_s$ to which the weight $a$ belongs we get  $\deg_{U_s}(w) \ge p-2 \ge 7$, 
a contradiction as before, provided that $p > 7$. For $p=7$ we get the same contradiction if $a \in \Phi(u_0)$, 
so it remains that $a= \theta(v_0)$, and therefore  $\Phi(m)= (a^6,5a,3a)$. 

(iii) if $m \not\in I_2 I_+$ and $\Phi(m_{U_t})=(a^i)$, $\Phi(m_{U_{l-t}})=(b^j)$ for some $a,b \in Z_p \setminus \{ 0\}$;
then $i >1$, say, hence by Lemma~\ref{lemma:easy}  a factorization $m=uv$ exists such that $u,v$
are $Z_p$-invariant monomials both containing a variable of weight  $a$;  we may also suppose that the zero-sum sequence $\Phi(v)$ is irreducible, 
hence $\deg(v) \le p$. Then $u\tau(v) \in I_+(R_+)_{\le p}$ and the monomials in $m - u\tau(v)$ 
 fall under case (i)-(ii).
\end{proof}

For the graded algebra $R=\F[V]^G$ and a positive integer $k$ 
the \emph{generalized Noether number} $\beta_k(R)$  was defined in \cite{cziszter-domokos:1} 
as the minimal positive integer $d$ such that $R_+$ is generated as a module over $R_+^k$ by elements of degree at most $d$. 
Evidently $\beta(R) = \beta_1(R)$.
The usefulness of this concept is shown in \cite{cziszter-domokos:1} and \cite{cziszter-domokos:2}.

\begin{theorem} \label{betak_zpz3}
Let $W = U_0 \oplus V^{\oplus 2}$ where $V$ is as in Corollary~\ref{cor:zpz3} 
and $U_0$ contains only  $1$-dimensional irreducible $G$-modules.  
Then $\beta_k(G, W) = kp+2$. 
If moreover $\ch(\F) = 0$ then $\beta_k(G) = kp+2$.
\end{theorem}

\begin{proof}
It suffices to prove that $\beta_k(G, W ) \le kp +2$ since the reverse
inequality follows from a more general statement in \cite{cziszter-domokos:2}. 
By a straightforward induction on $k$ it is enough to prove that $w \in I_+(R_+)_{\le p}$ holds
for any monomial $w \in I$ with $\deg(w) \ge p+3$, where  $I = \F[W]^A$ and $R= \F[W]^G$. 
We shall repeatedly use the fact that $I_+^3 \subseteq I_+R_+ + R_+$ (see \cite{cziszter-domokos:1}). 
Now, if $\deg_{U_0}(m) \ge 2$ then  $m \in I_1^2 I_{\ge  p+1} \subseteq I_+(R_+)_{\le p}$. 
If $\deg_{U_0}(m) = 1$ then $m \in I_1I_2I_+^2 +  I_+(R_+)_{\le p}$ by Proposition~\ref{ZpZ3_dupla}
and we are done again. 
Finally, if $\deg_{U_0}(m) = 0$ then either $m \in I_2^2I_+^2 + I_+(R_+)_{\le p}$ by applying Proposition~\ref{ZpZ3_dupla} twice,
in which case we are done, or else $p=7$ and $m=uv$ where $u \in I_2$ and $\Phi(v) = (a^6,5a,3a)$. 
In this latter case denoting by $s$ the index for which the weight $a$ belongs to $U_s$ we have
$\deg_{U_s}(m) = \deg_{U_s}(u) + \deg_{U_s}(v) = 7$, hence $\deg_{V_{s,i}}(m) \ge 4$ for some $i \in \{1,2 \}$, 
and consequently $m \in I_+(R_+)_{\le p}$ holds by Proposition~\ref{ZpZ3_borzalom}. 

The second claim follows from the first using a theorem of Weyl (see \cite{weyl} II. 5 Theorem 2.5A) stating that
for any unimodular $G$-modules $V_1,...,V_t$ of dimensions $n_i := \dim(V_i)$ and any integers $m_i \ge n_i$, where $i=1,...,t$,
the ring $\F[V_1^{\oplus m_1} \oplus ... \oplus V_t^{\oplus m_t}]$ is generated by the polarizations of a generating system of
$\F[V_1^{\oplus n_1-1} \oplus ... \oplus V_t^{\oplus n_t-1}]$ plus some invariants of degrees $n_1,...,n_t$,
whence in our case $\beta_k(G) = \beta_k(G,W)$ for all $k \ge 1$.
\end{proof}

\begin{remark}
Studying  the dependence of $\beta_k(G)$ on $k$ lead in \cite{cziszter-domokos:3} to focus on 
$\sigma(R)$ defined as the minimal positive integer $n$ 
such that $R$ is finitely generated as a module over its subalgebra $\F[R_{\le n}]$ generated by the elements of degree at most $n$. 
Moreover, $\eta(R)$ is the minimal $d$ such that $R_+$ is generated as an $\F[R_{\le \sigma(R)}]$-module by its elements of degree at most $d$. In \cite{cziszter-domokos:3} it was  proved that $\sigma(Z_p \rtimes Z_q) = p $ for any odd primes $p,q$ such that $q \mid p-1$. 
Therefore  the first half of the proof of Theorem~\ref{betak_zpz3} can be rephrased as stating that $\eta(Z_p \rtimes Z_3) \le p+2$. 
\end{remark}


\section{An  application for separating invariants}

A subset  $T \subset \F[V]^G$ is called a \emph{separating system} if  
for any $u,v \in V$ belonging to different $G$-orbits there is an element $f \in T$ such that $f(u) \neq f(v)$. 
A generating system of $\F[V]^G$ is a separating system, but the converse is not true. 
The study of separating systems was propagated in \cite{derksen-kemper}, and is in the focus of several recent papers, see e.g. \cite{d:sep}, \cite{grosshans:2007}, \cite{dufresne-elmer-kohls}, \cite{neusel-sezer}, \cite{domokos-szabo}. 
Following \cite{kohls-kraft} we write $\beta_{\mathrm{sep}}(\F[V]^G)$ for the minimal $n$ such that there exists a separating system in $\F[V]^G$ consisting of elements of degree at most $n$. Moreover, define $\beta_{\mathrm{sep}}(G)$ as the maximum of $\beta_{\mathrm{sep}}(\F[V]^G)$, where $V$ ranges over the isomorphism classes of $G$-modules. Similarly $\sigma(G)$ is defined in \cite{cziszter-domokos:3} as the maximal possible value of $\sigma(\F[V]^G)$. 
The following inequalities are well known: 
\begin{align}\label{beta_sep_triv}
\sigma(G) \le \beta_{\mathrm{sep}} (G) \le \beta(G)
\end{align}
For our group $G=Z_p\rtimes Z_3$  we have $\sigma(G)=p$ and $\beta(G)\ge p+2$ by \cite{cziszter-domokos:3} and \cite{cziszter-domokos:2}, and now we shall see that both  of the inequalities in \eqref{beta_sep_triv} are strict: 

\begin{theorem}\label{thm:sepZpZ3} 
We have $\beta_{\mathrm{sep}}(Z_p\rtimes Z_3)=p+1$. 
\end{theorem} 

\begin{proof} 
First we will prove the lower bound $\beta_{\mathrm{sep}}(G)\ge p+1$. 
Recall that $G = \bra a,b: a^p=b^3=1, bab^{-1}=a^r\ket$, where $r$ has order $3$ modulo $p$. 
Let $U$ and $V$ be irreducible representations of $G$ of dimension $1$ and $3$, respectively. 
Then $\F[U \oplus V] = \F[y,x_1,x_2,x_3]$ where $y^a=y$ and $y^b=\omega y$ for a primitive third root of unity $\omega$, 
while the $x_1,x_2,x_3$ are $a$-eigenvectors  of eigenvalues $\varepsilon, \varepsilon^r, \varepsilon^{r^2}$ for some primitive $p$-th root of unity $\varepsilon$,
and $x_1,x_2,x_3$ are cyclically permuted by $b$. 
Now, the point $(1,1,0,0)$ has trivial stabilizer in $G$, hence 
the points $(1,1,0,0)$ and $(\omega,1,0,0)$ do not belong to the same $G$-orbit. 
We claim that they cannot be separated by invariants of degree at most $p$. Indeed, suppose to the contrary that they can be separated. Then there exists an $\langle a\rangle$-invariant monomial $u$ with $\deg(u)\le p$ and $\tau(u)(1,1,0,0)\neq \tau(u)(\omega,1,0,0)$. If an $\langle a\rangle$-invariant  monomial $v$ involves at least two variables from $\{x_1,x_2,x_3\}$, then $\tau(v)$ vanishes on both of $(1,1,0,0)$ and $(\omega,1,0,0)$. If $u$ involves only $y$, then  $u=y^{3k}$ for some positive integer $k$, whence $\tau(u)$ takes the value $1$ both on $(1,1,0,0)$ and $(\omega,1,0,0)$. So $u$ involves exactly one variable from $\{x_1,x_2,x_3\}$, forcing that $u=x_i^p$, and then 
$\tau(u)$ agrees on the two given points, a contradiction.  

Next we prove the inequality $\beta_{\mathrm{sep}}(G)\leq p+1$. 
Let $W$ be the  multiplicity free $G$-module which contains every irreducible $G$-module  with multiplicity $1$. 
An arbitrary $G$-module is a direct summand in the direct sum $W^{\oplus n}$ of $n$ copies of $W$ for a sufficiently large integer $n$. 
According to a theorem of Draisma, Kemper, and Wehlau  \cite{draisma-kemper-wehlau} (see also \cite{losik-michor-popov} and \cite{grosshans:2007})  a separating set of $\F[W^{\oplus n}]^{G}$
can be obtained by ``cheap" polarization from a separating set of $\F[W]^G$, which is a degree-preserving procedure, whence 
 $ \beta_{\mathrm{sep}}(\F[W^{\oplus n}]^G) = \beta_{\mathrm{sep}}(\F[W]^G)$ and consequently 
\begin{align} \beta_{\mathrm{sep}}(G) = \beta_{\mathrm{sep}}(\F[W]^G) \end{align}
Now let $W= U \oplus V$ 
where $U$ is the sum of $1$-dimensional irreducibles, and $V$ is the sum of $3$-dimensional irreducibles. 
Suppose that $(u_1,v_1)$ and $(u_2,v_2)$ belong to different $G$-orbits in $U\oplus V$. We need to show that they can be separated by a polynomial invariant of degree at most $p+1$. 
If $u_1$ and $u_2$ belong to different $G$-orbits, then they can be separated by a polynomial invariant of degree at most $3=|G/A|$ (recall that $A$ acts trivially on $U$), and we are done. From now on we assume that $u_1$ and $u_2$ have the same $G/A=B$-orbits. 
If $v_1$ and $v_2$ belong to different $G$-orbits in $V$, then they can be separated by a $G$-invariant on $V$ of degree at most $p$ by 
Corollary~\ref{cor:zpz3} and we are done. It remains that $v_1=gv_2$ for some $g\in G$; after replacing $(u_2,v_2)$ by an appropriate element $v_1=v_2=v$ might be assumed, and then $u_1$ and $u_2$ are not on the same orbit under the stabilizer  $\mathrm{Stab}_G(v)$ of $v$. Consequently $\mathrm{Stab}_G(v)$ is contained in $ A$ (for otherwise $\mathrm{Stab}_G(v)$ is mapped surjectively onto $G/A$). 
Let $U_0$ be a $1$-dimensional summand in $U=U_0\oplus U_{1}$ such that $\pi(u_1,v)\neq\pi(u_2,v)$, where $\pi: U \oplus V \to U_0$ is the projection onto $U_0$ with kernel 
$U_{1} \oplus V$. Denote by $y$ the coordinate function on $U_0$, viewed as an element of $\mathbb{F}[U \oplus V]$ by composing it with $\pi$. If $G$ acts trivially on $U_0$, then $y$ is a $G$-invariant of degree $1$ that separates $(u_1,v)$ and $(u_2,v)$. Otherwise $y^g=\chi(g)^{-1}y$ for some non-trivial group homomorphism  
$\chi:G\to\F^\times$. 
By Lemma~\ref{lemma:zerolocusofrelinv} below, there is an $f\in\mathbb{F}[V]^{G,\chi}:=\{h\in\F[V]\mid h^g=\chi(g)h \quad \forall g\in G\}$ such that $f(v)\neq 0$. Moreover, since a twisted version of the relative transfer map gives an $R$-module surjection $I\to \F[V]^{G,\chi}$ 
(see in the proof of Lemma~\ref{lemma:zerolocusofrelinv}) we infer from Corollary~\ref{cor:zpz3} that
$\mathbb{F}[V]^{G,\chi}$ is generated as an $\mathbb{F}[V]^G$-module by its elements of degree at most $p$, 
hence we may assume that $f$  has degree at most $p$. 
Now $yf$ is a $G$-invariant of degree at most $p+1$ which separates $(u_1,v)$ and $(u_2,v)$  by construction. 
\end{proof}

\begin{remark}
Note that by a straightforward extension of the argument in the first part of the above proof
we get  $\beta_{\mathrm{sep}}(Z_p \rtimes Z_q) \ge p+1$ for any $q \mid p-1$.
\end{remark}

\begin{lemma} \label{lemma:zerolocusofrelinv}
Given a group homomorphism $\chi: G \to \mathbb{F}^{\times}$ and any point $v \in V$ such that $\Stab_G(v) \le \ker(\chi)$,
a relative invariant $f \in \mathbb{F}[V]^{G,\chi}$ exists for which $f(v) \neq 0$.
\end{lemma}

\begin{proof}
Let $g_1,...,g_n \in G$ be a system of representatives of the left cosets of the subgroup $\Stab_G(v)$. 
Then  $g_1  v, ..., g_n v \in V$ are distinct. 
Therefore there exists a polynomial $h \in \mathbb{F}[V]$ such that
$h(g_i  v) = \chi(g_i)$   for each $i=1,...,n$. 
Now  set $f := \tau_{\chi}(h)$ where $\tau_{\chi}: \mathbb{F}[V] \to \mathbb{F}[V]^{G,\chi}$
is the twisted transfer map defined as
$\tau_{\chi}(h):=\sum_{g\in G}\chi(g)^{-1}h^g$. 
Then $f\in \F[V]^{G,\chi}$ by construction.
Moreover, 
since $\Stab_G(v) \le \ker(\chi)$  we have 
\[ f(v) = \sum_{g \in G} \chi(g)^{-1} h(g  v) = |\Stab_G(v)| \sum_{i=1}^n \chi(g_i)^{-1} h(g_i  v) = |G| \]
which is indeed a non-zero element in $\mathbb{F}$. 
\end{proof}

\begin{center} {\bf Acknowledgement}\end{center}

The author thanks M\'aty\'as Domokos for conversations that helped to complete this research.



\begin{thebibliography}{mmm}

\bibitem{balandraud} E. Balandraud, An addition theorem and maximal zero-sum free sets in $\mathbb{Z}/p\mathbb{Z}$, 
 Israel Journal of Mathematics 188 (2012), 405-429. 



\bibitem{benson} D. J. Benson, Polynomial Invariants of Finite Groups, 
Cambridge University  Press, 1993. 


\bibitem{cziszter-domokos:1} K. Cziszter and M. Domokos, Groups with large  Noether bound, arXiv:1105.0679v4 .

\bibitem{cziszter-domokos:2} K. Cziszter and M. Domokos, Noether's bound for the groups with a cyclic subgroup of index two, arXiv:1205.3011 

\bibitem{cziszter-domokos:3} K. Cziszter and M. Domokos, On the generalized Davenport constant and the Noether number, arXiv:1205.3416

\bibitem{derksen-kemper} 
H. Derksen and G. Kemper, 
Computational Invariant Theory, 
Springer-Verlag, Berlin, 2002. 

\bibitem{d:sep} M. Domokos, 
Typical separating invariants, Transform. Groups 12 (2007), 49-63. 


\bibitem{domokos-hegedus} M. Domokos and  P. Heged\H us, 
Noether's bound for polynomial invariants of finite groups, 
Arch. Math. 74 (2000), 161-167. 

\bibitem{domokos-szabo} M. Domokos and E. Szab\'o, Helly dimension of algebraic groups, 
J. Lond. Math. Soc., II. Ser. 84 (2011), 19-34. 

\bibitem{draisma-kemper-wehlau} J. Draisma, G. Kemper, and D. Wehlau, 
Polarization of separating invariants, Canad. J. Math.  60 (2008), no. 3, 556-571.

\bibitem{dufresne-elmer-kohls}  E. Dufresne, J. Elmer and M. Kohls, 
The Cohen-Macaulay property of separating invariants of finite groups, 
Transform. Groups 14 (2009),  771-785.


\bibitem{fleischmann} P. Fleischmann, The Noether bound in invariant theory of finite groups, Adv. Math. 156 (2000), 23-32. 

\bibitem{fogarty} J. Fogarty, On Noether's bound for polynomial invariants of a finite group, Electron. Res. Announc. Amer. Math. Soc. 7 (2001), 5-7. 

\bibitem{freeze-smith} M. Freeze and W. W. Smith, Sumsets of zerofree sequences, Arab J. Sci. Eng. Section C: Theme Issues 26 (2001), 97-105.

\bibitem{geroldinger-gao} W. Gao and A. Geroldinger, Zero-sum problems in finite abelian groups: A survey, 
Expo. Math. 24 (2006), 337-369. 

\bibitem{geroldinger-halterkoch}  A. Geroldinger and F. Halter-Koch, Non-unique factorizations. Algebraic, combinatorial and analytic theory, Monographs and Textbooks in Pure and Applied Mathematics, Chapman \& Hall/CRC, 2006.

\bibitem{grosshans:2007}  F. D. Grosshans, 
Vector invariants in arbitrary characteristic, 
Transform. Groups 12 (2007), 499-514. 


\bibitem{kohls-kraft} M. Kohls and H. Kraft, Degree bounds for separating invariants, Math. Res. Lett. 17 (2010), 1171-1182. 

\bibitem{losik-michor-popov} M. Losik, P. W. Michor, and  V. L. Popov, On polarizations in invariant theory, J. Algebra
301 (2006), 406-424.

\bibitem{neusel-sezer}  M. D. Neusel and M. Sezer, 
Separating invariants of modular p-groups and groups acting diagonally, 
Math. Res. Letters 16 (2009), 1029-1036. 


\bibitem{noether:1916} E. Noether, Der Endlichkeitssatz der Invarianten endlicher Gruppen, 
Math. Ann. 77 (1916), 89-92. 

\bibitem{pawale} Vivek M. Pawale, Invariants of semidirect products of cyclic groups, Ph.D. Thesis, Brandeis University, 1999.


\bibitem{schmid} B. J. Schmid, Finite groups and invariant theory, 
``Topics in Invariant Theory'', Lect. Notes in Math. 1478 (1991), 
35-66. 

\bibitem{sezer} M. Sezer, Sharpening the generalized Noether bound in the invariant theory of finite groups, J. Algebra 254 (2002), 252-263. 

\bibitem{tao} T. Tao and V. Vu, Additive Combinatorics, Number 105 in Cambridge
Studies in Advanced Mathematics, Cambridge University Press, 2006.

\bibitem{wehlau} D. L. Wehlau, The Noether number in invariant theory, Comptes Rendus Math. Rep. Acad. Sci. Canada Vol. 28 No. 2 (2006), 39-62.

\bibitem{weyl} H. Weyl, The Classical Groups, their Invariants and Representations, Princeton Mathematical Series, vol. 1,
Princeton Univ. Press, Princeton 1946

\end{thebibliography}
\end{document}